%% file: CUR_IMRN.tex
\documentclass[12pt]{amsart}
\usepackage[top=1in,left=1in,bottom=1in,right=1in]{geometry}

\include{neuralcodecommands}

\begin{document}
\title{Convex union representability and convex codes}

\author{R. Amzi Jeffs}
\address{Department of Mathematics.  University of Washington, Seattle, WA 98195}
\email{rajeffs@uw.edu}

\author{Isabella Novik}
\address{Department of Mathematics.  University of Washington, Seattle, WA 98195}
\email{novik@uw.edu}

\begin{abstract} We introduce and investigate $d$-convex union representable complexes: the simplicial complexes that arise as the nerve of a finite collection of convex open sets in $\R^d$ whose union is also convex. Chen, Frick, and Shiu recently proved that such complexes are collapsible and asked if all collapsible complexes are convex union representable. We disprove this by showing that there exist shellable and collapsible complexes that are not convex union representable; there also exist non-evasive complexes that are not convex union representable. In the process we establish several necessary conditions for a complex to be convex union representable such as: that such a complex $\Delta$ collapses onto the star of any face of $\Delta$, that the Alexander dual of $\Delta$ must also be collapsible, and that if $k$ facets of $\Delta$ contain all free faces of $\Delta$, then $\Delta$ is $(k-1)$-representable. We also discuss some sufficient conditions for a complex to be convex union representable. The notion of convex union representability is intimately related to the study of convex neural codes. In particular, our results provide new families of examples of non-convex neural codes.
\end{abstract}

\thanks{Jeffs' research is partially supported by  graduate fellowship from NSF grant DMS-1664865. Novik's research is partially supported by NSF grant DMS-1664865 and by Robert R.~\& Elaine F.~Phelps Professorship in Mathematics}
\date{\today}
\maketitle

\section{Introduction}\label{sec:intro}

The goal of this paper is to initiate the study of convex union representable complexes --- a certain subfamily of representable complexes. An abstract simplicial complex is called $d$-representable if it is the nerve of a family of convex sets in $\R^d$.  The research on $d$-representable complexes has a rich and fascinating history starting with Helly's theorem (see for instance the survey articles \cite{eckhoff, tancer}  and the references therein), yet for $d>1$, the problem of characterizing $d$-representable complexes remains wide open.

We say that a simplicial complex is \emph{$d$-convex union representable} if it arises as the nerve of a finite collection of convex open sets in $\R^d$ whose union is also convex, and that a complex is \emph{convex union representable} if it is  $d$-convex union representable for some $d$. (It is worth stressing right away that convex union representability does not change if we replace the openness requirement with closedness, see Proposition \ref{prop:polytoperealization}.)

Our motivation for investigating such complexes  comes from the theory of convex neural codes --- a much younger subject, see for instance \cite{undecidability, chadvlad, local15, neuralring13, nogo, obstructions}. We defer precise definitions until later sections, and for now only mention that a convex neural code $\C$ is a combinatorial code that records the regions cut out  by a collection of convex open sets in some Euclidean space. The smallest simplicial complex that contains $\C$, denoted $\Delta(\C)$, is called the simplicial complex of $\C$. If $\C$ is a convex neural code and $\sigma$ is an element of  $\Delta(\C)$ but not of $\C$, then the link of $\sigma$ in $\Delta(\C)$ must be convex union representable, see Section \ref{sec:codes}. Thus, to shed light on the fundamental question of which (combinatorial) codes are convex, it is imperative to deepen our understanding of convex union representable complexes --- the task we begin in this paper. 

It follows from Borsuk's nerve lemma \cite[Theorem 10.6]{bjorner95} that every convex union representable complex is acyclic, and even contractible. In \cite[Section 5]{undecidability}, Chen, Frick, and Shiu used a technique of \cite{wegner} to show that, in fact, convex union representable complexes are collapsible. They also raised the question of whether all collapsible complexes are convex union representable. We disprove this, and in the process
establish several further necessary conditions for a complex to be convex union representable. Among our results are the following:

\begin{itemize}
\item A convex union representable complex collapses onto the star of \emph{any} of its faces; see Corollary \ref{cor:collapsetostar}.
\item The Alexander dual of a convex union representable complex is also collapsible; see Theorem \ref{thm:dualcollapsible}.
\item Convex union representable complexes are similar in spirit to constructible complexes; see  Theorem \ref{thm:splitting} for a precise statement.
\item  A convex union representable complex that has $k$ or fewer free faces is $(k-1)$-convex union representable. In particular, it is $(k-1)$-representable, and hence also $(k-1)$-Leray; see Theorem \ref{thm:(k-1)rep}.
\item For every $d\geq 2$, there exists a $d$-dimensional shellable and collapsible simplicial complex that is not convex union representable. Similarly, for  every $d\geq 2$, there exists a $d$-dimensional non-evasive complex that is not convex union representable; see  Corollary \ref{cor:nonevasivebutnotrepresentable} for both results. On the other hand,  a $1$-dimensional complex is convex union representable if and only if it is collapsible, which happens if and only if it is a tree; see Corollary \ref{cor:tree}.
\end{itemize}

The structure of the rest of the paper is as follows. In Section \ref{sec:preliminaries}, we set up basic notation as well as review some necessary background related to simplicial complexes, polytopes, nerves, and representability. In Section \ref{sec:collapsiblenotCUR} we discuss several properties that convex union representable complexes possess and use them to construct examples of collapsible complexes that are not convex union representable. Section \ref{sec:proofs} provides proofs of some auxiliary results; this is the most technical part of the paper. These results are then used in Sections \ref{sec:construtcable-like}--\ref{sec:fewfreefaces} to establish several further interesting properties that convex union representable complexes satisfy such as constructible-like behavior, collapsibility of the Alexander dual, etc. These properties in turn lead to additional examples of collapsible complexes that are not convex union representable. Section \ref{sec:constructions} discusses a few sufficient conditions for a complex to be convex union representable. In Section \ref{sec:codes} we describe how our results apply to the theory of convex neural codes. We conclude in Section \ref{sec:conclusion} with some open questions. 
 
\section{Preliminaries} \label{sec:preliminaries}
\subsection{Simplicial complexes}
We begin with several basic definitions pertaining to abstract simplicial complexes. From here on out we omit ``abstract" for brevity. For all undefined terminology we refer the readers to \cite{bjorner95}. A \emph{simplicial complex} $\Delta$ on a finite ground set $\V$ is a collection of subsets of $\V$ that is closed under inclusion.  The elements of $\Delta$ are called \emph{faces}, the inclusion-maximal faces are called \emph{facets}, and $1$-element faces are \emph{vertices}. (An element of the ground set may not form a vertex.) The \emph{dimension} of a face $\sigma\in\Delta$ is the cardinality of $\sigma$ minus one, and the dimension of $\Delta$ is the maximum dimension of its faces.  If $\sigma$ is any subset of $\V$, then the simplex on $\sigma$ is the collection $\dc{\sigma}\od 2^\sigma$ (i.e., the down-closure of $\sigma$).

We regard the empty set as a face. Also, following standard combinatorial conventions, see for instance \cite[p.~20]{St96}, we distinguish between the \emph{void} simplicial complex $\emptyset$ that has no faces, and the \emph{empty} simplicial complex $\{\emptyset\}$ that has a single empty face. One reason for such a distinction is that the void complex is acyclic, i.e., all its reduced homology vanishes, while the empty complex is not: it has a non-zero $\tilde{H}_{-1}$. (In the same vein, the Alexander dual of a nonempty simplex $\dc{\sigma}$ is $\emptyset$, while the Alexander dual of the boundary of $\dc{\sigma}$ is $\{\emptyset\}$.)

A simplicial complex $\Delta$ gives rise to several new simplicial complexes. The \emph{restriction} of $\Delta$ to $\omega\subseteq \V$ (or the \emph{induced subcomplex} of $\Delta$ on $\omega$) is $\Delta|_\omega:=\{\sigma\in\Delta \mid \sigma\subseteq \omega\}$. 
If $\sigma\in \Delta$, then the \emph{star} and the \emph{link} of $\sigma$ in $\Delta$ are the following subcomplexes of $\Delta$:
\[ 
\cstar_\Delta(\sigma)=\cstar(\sigma):=\{\tau\in\Delta \mid \tau\cup\sigma\in \Delta\}, \quad 
\link_\Delta(\sigma)=\link(\sigma):=\{\tau\in\cstar_\Delta(\sigma)\mid \tau\cap\sigma =\emptyset\}.
\]
Similarly, the \emph{contrastar} of $\sigma$ in $\Delta$ (also known as the \emph{deletion} of $\sigma$ from $\Delta$) is the subcomplex $\Delta\setminus\sigma:=\{\tau\in\Delta \mid \tau\not\supseteq\sigma\}$. If $v$ is a vertex of $\Delta$, it is customary to write $v\in \Delta$, $\cstar(v)$, $\link(v)$, and $\Delta\setminus v$ instead of $\{v\}\in\Delta$, $\cstar(\{v\})$, $\link(\{v\})$, and $\Delta\setminus \{v\}$, respectively. 

The join of two simplicial complexes $\Delta_1$ and $\Delta_2$ with disjoint vertex sets is the complex
\[
\Delta_1\ast\Delta_2:=\{\sigma_1\cup\sigma_2\mid \sigma_1\in \Delta_1, \sigma_2\in \Delta_2\}.
\]
Note that if $\Delta$ is a simplicial complex and $\sigma\in\Delta$, then $\cstar_\Delta(\sigma)=\dc{\sigma} \ast \link_\Delta(\sigma)$. 

Of a particular interest are the \emph{cone} and the \emph{suspension} of $\Delta$.  The cone over $\Delta$ with apex $v$ is the join of $\Delta$ with a $0$-dimensional simplex $\dc{v}$, where $v$ is not a vertex of $\Delta$; it is customary to denote this complex by $v\ast\Delta$ instead of  $\dc{v}\ast\Delta$. The suspension $\Sigma\Delta$ of $\Delta$ is the join of $\Delta$ with the zero-dimensional sphere. The two vertices of that sphere are called the \emph{suspension vertices}.

A pair of simplicial complexes $\Delta\supseteq\Gamma$ gives rise to a \emph{relative simplicial complex} denoted $(\Delta,\Gamma)$ and defined as the set-theoretic difference of $\Delta$ and $\Gamma$. In other words, \emph{faces} of $(\Delta,\Gamma)$ are faces of $\Delta$ that are not faces of $\Gamma$.  If $\emptyset \neq \Gamma\subsetneq \Delta$, then $(\Delta,\Gamma)$ is \textbf{not} a simplicial complex. 

The following notion is due to Whitehead. Assume $\Delta$ is a simplicial complex, and $\sigma$ is a face of $\Delta$ that is not a facet. If $\sigma$ is contained in a \emph{unique} facet of $\Delta$, then we say that $\sigma$ is a \emph{free face} of $\Delta$. In such a case, the operation $\Delta\rightarrow \Delta\setminus\sigma$ is called an \emph{elementary collapse} of $\Delta$ induced by $\sigma$.  A sequence of elementary collapses starting with $\Delta$ and ending with $\Gamma$ is a \emph{collapse} of $\Delta$ onto $\Gamma$. 

We say that $\Delta$ is \emph{collapsible} if it collapses to  the void complex $\emptyset$. Observe that if $\Delta$ is a simplicial complex with at least one vertex, then the empty face of $\Delta$ is a free face if and only if $\Delta$ is a simplex. In particular, a complex with at least one vertex is collapsible if and only if it collapses to a nonempty simplex which happens if and only if it collapses to a single vertex.  It is known that all nonempty collapsible complexes are contractible. 

A related notion, introduced in \cite{topologicalevasiveness} (see also \cite[Section 11]{bjorner95}), is that of a \emph{non-evasive complex}. It is worth mentioning that any cone is a non-evasive complex, and that any non-evasive complex  is collapsible.

\subsection{Polytopes}
A \emph{polytope} $P$ in $\R^d$ is the convex hull of finitely many (possibly zero) points in $\R^d$. Equivalently, a polytope is a compact subset of $\R^d$ that can be written as the intersection of finitely many closed half-spaces. The \emph{dimension} of $P$ is defined as the dimension of the affine span of $P$. The interior of a polytope (as well as of any subset) $P\subseteq \R^d$ is denoted by $\interior(P)$. A (proper) face of a $d$-dimensional polytope $P\subset \R^d$ is the intersection of $P$ with a supporting hyperplane in $\R^d$. 
We say that $P$ is a \emph{simplicial polytope} if all proper faces of $P$ are geometric simplices (i.e., convex hulls of affinely independent points). In this case, there is an associated simplicial complex --- the boundary complex of $P$, denoted $\partial P$. The facets of $\partial P$ are the vertex sets of the maximal faces of $P$ and the dimension of $\partial P$ is $\dim P-1$.  We refer the readers to Ziegler's book \cite{ziegler-polytopesbook} for an excellent introduction to the theory of polytopes.

\subsection{Minkowski sums and distances}
For two subsets $X,Y$ of $\R^d$, their \emph{Minkowski sum} is $X+Y:= \{x+y \mid x\in X, \ y\in Y\}$. (Note that $\emptyset + Y =\emptyset$.) 
We will often consider the Minkowski sum of $X$ and the open ball $B_\varepsilon$ of radius $\varepsilon$ centered at the origin.

For $\emptyset\neq X, Y\subseteq\R^d$, the \emph{distance} between $X$ and $Y$ is defined as $\inf\{\|x-y\| \mid x\in X, y\in Y\}$. Note that the distance between $X$ and $Y$ is at least $\varepsilon$ if and only if $X+B_\varepsilon$ is disjoint from $Y$. In Lemma \ref{lem:shrink} (and only there), we will also consider a very different notion --- the \emph{Hausdorff distance} $\dist_H(X,Y)$ between bounded subsets $X$ and $Y$of $\R^d$; it is defined as 
\[\dist_H(X,Y):=\inf\{\varepsilon\geq 0 \mid X\subseteq Y+B_\varepsilon \mbox{ and } Y\subseteq X+B_\varepsilon\}.
\] 
The Hausdorff distance makes the space of compact subsets of $\R^d$ into a metric space. 

\subsection{Nerves and representability}
Let $[n]$ denote the set of integers $\{1,2,\ldots, n\}$.
We will often work with collections $\{U_1,\ldots, U_n\}$ of convex subsets of $\R^d$. (The empty set is considered convex; it is open and closed.) When working with such a collection we let $U_\sigma$ denote $\bigcap_{i\in\sigma} U_i$ for all $\sigma\subseteq [n]$, with the convention that $U_\emptyset=\conv(\bigcup_{i=1}^n U_i)$. 
The \emph{nerve} of a collection of sets $\U=\{U_1,\ldots, U_n\}$ is the simplicial complex 
\[
\nerve(\U) \od \{\sigma\subseteq [n]\mid U_\sigma \neq \emptyset\}. 
\]
In particular, if all $U_i$ are empty sets, then $\nerve(\U)=\emptyset$. As a result, $\nerve(\U)$ is never $\{\emptyset\}$.

A simplicial complex $\Delta\subseteq 2^{[n]}$ is called \emph{$d$-representable} if there exists a collection of convex sets $\U=\{U_1,\ldots, U_n\}$ in $\R^d$ such that $\Delta = \nerve(\U)$. The class of $d$-representable complexes is well-studied, see for instance \cite{tancer} and \cite{wegner}. In particular, it is known that (i) every finite simplicial complex is $d$-representable for $d$ large enough, and that (ii) every $d$-representable complex is $d$-Leray.

In this paper we initiate the study of the following subfamily of $d$-representable complexes.

\begin{definition}\label{def:convexunionrepresentable}
Let $\Delta\subseteq 2^{[n]}$ be a simplicial complex. We say that $\Delta$ is \emph{$d$-convex union representable} if there is a collection of convex \emph{open} sets $\U=\{U_1,\ldots, U_n\}$ in $\R^d$ such that \begin{itemize}
\item[(i)] $\bigcup_{i\in[n]} U_i$ is a convex open set, and 
\item[(ii)] $\Delta = \nerve(\U)$. 
\end{itemize}
We say that $\Delta$ is \emph{convex union representable} if there exists some $d$ such that $\Delta$ is $d$-convex union representable. The collection $\{U_i\}_{i=1}^n$ is called a \emph{$d$-convex union representation} of $\Delta$. 
\end{definition}

\section{Collapsible complexes that are not convex union representable} \label{sec:collapsiblenotCUR}
In this section we discuss some of the properties that convex union representable complexes satisfy. We then use these properties to show the existence of collapsible complexes that are not convex union representable. Our starting point is \cite[Lemma 5.9]{undecidability} asserting 
that convex union representable complexes are collapsible. Here we further strengthen this result. Our main tool in doing so is the following theorem, whose proof is a consequence of the methods used in \cite{undecidability} and \cite{wegner}. 

\begin{theorem}\label{thm:halfspace}
Let $\{U_i\}_{i=1}^n$ be a $d$-convex union representation of a simplicial complex $\Delta$, and let $H\subseteq \R^d$ be a closed or open halfspace of $\R^d$. Then $\Delta$ collapses onto $\nerve\big(\{U_i\cap H\}_{i=1}^n\big)$. 
\end{theorem}

Note that if $H$ is disjoint from all $U_i$, then $\nerve\big(\{U_i\cap H\}_{i=1}^n\big)=\emptyset$. In this case Theorem \ref{thm:halfspace} simply says that $\Delta$ is collapsible, and hence recovers \cite[Lemma 5.9]{undecidability}.

From Theorem \ref{thm:halfspace}, whose proof we postpone until the next section, a number of remarkable results (including the promised examples of collapsible but not convex union representable complexes) follow quickly. Below we summarize several.

\begin{corollary}\label{cor:collapsetoconvexset}
Let $\{U_i\}_{i=1}^n$ be a $d$-convex union representation of a simplicial complex $\Delta$, and let $C\subseteq \R^d$ be a convex set. Then $\Delta$ collapses onto $\nerve\big(\{U_i\cap C\}_{i=1}^n\big)$. 
\end{corollary}
\begin{proof}
For all $U_\sigma$ such that $U_\sigma\cap C\neq \emptyset$, choose a point $p_\sigma\in U_\sigma\cap C$, and let $C'$ be the convex hull of these points. Observe that $C'$ is a polytope contained in $C$ such that $\nerve(\{U_i\cap C\}) = \nerve(\{U_i\cap C'\})$. Since $C'$ is the intersection of finitely many closed halfspaces, we can repeatedly apply Theorem \ref{thm:halfspace} to obtain that $\Delta$ collapses to $\nerve(\{U_i\cap C'\}_{i=1}^n)$, proving the result. 
\end{proof}

\begin{corollary}\label{cor:collapsetostar}\label{cor:collapsetosimplex}
Let $\Delta$ be a convex union representable complex, and let $\sigma\in\Delta$ be an arbitrary face. Then $\Delta$ collapses onto the star of $\sigma$. In particular, if $\sigma\neq\emptyset$, then $\Delta$ collapses onto $\dc{\sigma}$. 
\end{corollary}
\begin{proof} Let $\{U_i\}_{i=1}^n$ be a $d$-convex union representation of $\Delta$, and let $C=U_\sigma$. Then $C$ is a nonempty convex subset of $\R^d$. Therefore, by Corollary \ref{cor:collapsetoconvexset}, $\Delta$ collapses onto $\nerve\big(\{U_i\cap C\}_{i=1}^n\big)$. The result follows since $\nerve(\{U_i\cap C\})=\cstar_\Delta(\sigma)$. Indeed, if $\tau\subseteq [n]$, then $\bigcap_{j\in\tau}(U_j\cap U_\sigma)=U_{\tau\cup\sigma}$. Thus, $\tau\in \nerve(\{U_i\cap C\})$ if and only if $\tau\cup\sigma\in\nerve(\{U_i\})=\Delta$, which happens if and only if $\tau\in\cstar_\Delta(\sigma)$.
\end{proof}


\begin{corollary}\label{cor:nocommonvertex}
Let $\Delta$ be a convex union representable complex. Then the free faces of $\Delta$ cannot all share a common vertex.  
\end{corollary}
\begin{proof}
Suppose the free faces of $\Delta$ share a common vertex $v$. Then no collapse of $\Delta$ other than $\Delta$ itself would contain $\cstar_\Delta(v)$. Since $\Delta$ collapses to $\cstar_\Delta(v)$, this implies $\Delta = \cstar_\Delta(v)$. Thus $\Delta$ is a cone over $v$. But then any facet of $\Delta\setminus v$ is a free face of $\Delta$. Such a free face does not contain $v$, a contradiction. 
\end{proof}

In \cite[Theorem 2.3]{onefreeface} the authors construct examples for all $d\ge 2$ of a $d$-dimensional collapsible simplicial complex $\Sigma_d$ with only one free face. According to Corollary \ref{cor:nocommonvertex}, these provide an example of collapsible complexes that are not convex union representable.  (The complex $\Sigma_2$ has only $7$ vertices, see \cite[Figure 2]{onefreeface}.) The authors also give examples for all $d\ge2$ of a $d$-dimensional simplicial complex $E_d$ which is pure and non-evasive,  has only two free faces, and, furthermore, these two free faces share a common ridge (see \cite[Theorem 2.5]{onefreeface}). By Corollary \ref{cor:nocommonvertex} these complexes are not convex union representable either. We formalize these observations in the following corollary.

\begin{corollary}\label{cor:collapsiblebutnotrepresentable}\label{cor:nonevasivebutnotrepresentable}
The simplicial complexes $\Sigma_d$ of \cite{onefreeface} are pure, collapsible, and shellable, but not convex union representable. Similarly, the simplicial complexes $E_d$ of \cite{onefreeface} are pure and non-evasive, but not convex union representable.
\end{corollary}


\section{Proof of Theorem \ref{thm:halfspace}}\label{sec:proofs}
The goal of this section is to verify Theorem \ref{thm:halfspace}. This result is then used in the rest of the paper to establish several additional necessary conditions for a simplicial complex to be convex union representable. We begin with a lemma that allows us to modify a given convex union representation. While it is clear that one can slightly shrink the sets of a convex union representation to preserve the nerve, some care is needed to make sure that this shrinking also preserves convexity of the union (e.g., that it does not create ``holes"). In the following, for $V\subseteq \R^d$, we denote by $\ol{V}$ the closure of $V$. 

\begin{lemma}\label{lem:shrink}
Let $\{U_i\}_{i=1}^n$ be a convex union representation of a simplicial complex $\Delta$, where all $U_i$ are bounded. Then for all $\varepsilon >0$ there exists a convex union representation $\{V_i\}_{i=1}^n$ of $\Delta$ with the following properties:
\begin{itemize}
\item[(i)] $\ol{V_i}$ is a polytope contained in $U_i$ for all $i\in[n]$,
\item[(ii)] The union $\bigcup_{i=1}^n V_i$ is the interior of a polytope, and 
\item[(iii)] The Hausdorff distance between $\ol{V_\sigma}$ and $\ol{U_\sigma}$ is less than $\varepsilon$ for all $\emptyset\neq\sigma\in \Delta$. 
\end{itemize} 
\end{lemma}
\begin{proof}
The result holds if all $U_i$ are empty sets: simply take all $V_i$ to be empty. Thus, assume $\Delta$ has at least one vertex, and for each face $\emptyset\neq\sigma\in \Delta$, fix a $d$-dimensional polytope $P_\sigma\subseteq U_\sigma$ with the property that $\dist_H\big(\overline{U_\sigma},P_\sigma\big)<\varepsilon/3$. Let $C$ be the convex hull of the union of all $P_\sigma$. Then $C$ is a $d$-dimensional polytope (hence compact) covered by $\{U_i\}_{i=1}^n$, so we may choose a Lebesgue number $\delta>0$ for this cover. 

Consider the lattice $(\delta\Z)^d$. For every point $p$ in this lattice, let $W_p$ be the closed $d$-cube with side length $2\delta$ centered at $p$. Then $W_p\cap C$ is a polytope. Let $S$ denote the collection of vertices of all nonempty polytopes of this form. Note that $S$ is finite and that $S$ might not be a subset of $(\delta\Z)^d$ since some of the cells of the lattice may only partially intersect $C$. 

 By shrinking $\delta$, we may assume that the following conditions hold:\begin{itemize}
\item[(1)] Every nonempty set of the form $U_\sigma\cap C$ contains some $W_p$, 
\item[(2)] for every $p\in (\delta\Z)^d$ with $W_p\cap C \neq \emptyset$, there exists $i\in[n]$ such that $W_p\subseteq U_i$, and
\item[(3)] the diameter of $W_p$ is less than $\varepsilon/3$.
\end{itemize}
For $i\in [n]$, define \[
V_i = \interior(\conv(S\cap U_i)), \mbox{ so that } \ol{V_i}=\conv(S\cap U_i)
\]
is a polytope contained in $U_i$. In the remainder of the proof, we use conditions (1)-(3) to show that $\{V_i\}_{i=1}^n$ is the desired convex union representation of $\Delta$.

By choice of $C$, $U_\sigma\cap C\neq \emptyset$ for every face $\emptyset\neq\sigma\in \Delta$. Condition (1) then implies that $U_\sigma\cap C$ contains $W_p$ for some $p \in (\delta\Z)^d$, which, in turn, implies that $p\in V_i$ for all $i\in \sigma$. Thus $\sigma\in \nerve(\{V_i\}_{i=1}^n)$. Since $V_i\subseteq U_i$ for all $i$, we conclude that $\nerve(\{V_i\}_{i=1}^n) = \Delta$.

To verify property (ii), we show that $\bigcup_{i=1}^n V_i=\interior C$. Indeed, $\interior C$ is the union of all sets in the collection $\{\interior W_p\cap\interior  C \neq \emptyset \mid p\in(\delta\Z)^d\}$. Furthermore, by condition (2), each polytope $W_p\cap C$ is contained in $U_i$ for some $i$, and so its vertices are in $U_i\cap S$. By definition of $V_i$ this implies that $\interior (W_p\cap C) \subseteq V_i$. The assertion follows since $\interior(W_p\cap C) = \interior W_p\cap \interior C$. 

For property (iii), note that $P_\sigma\subseteq \overline{U_\sigma}\cap C\subseteq \overline{U_\sigma}$, and so by choice of $P_\sigma$,
\[\dist_H\big(\overline{U_\sigma}, \overline{U_\sigma}\cap C)\leq \dist_H\big(\overline{U_\sigma}, P_\sigma)<\varepsilon/3.
\]
Also, $\dist_H\big(\overline{U_\sigma}\cap C, \overline{V_\sigma}\big)<\varepsilon/3$ by definition of the sets $V_i$ and by condition (3). Property (iii) follows since the Hausdorff distance is a metric on the space of compact subsets of $\R^d$.
\end{proof}

An interesting consequence of this lemma (see Proposition \ref{prop:polytoperealization}) is that convex union representability does not change if we replace the openness requirement with closedness. (However, dropping this requirement outright would change the notion: consider for example $U_1 = (0,1)$ and $U_2 = [1,2)$ in $\R^1$. The nerve of this collection is not even a connected simplicial complex, but $U_1\cup U_2$ is a convex subset of $\R^1$.) In contrast, this equivalence does not hold for convex realizations of neural codes; for details, see Remark \ref{rem:closed-vs-open} below.

\begin{proposition}\label{prop:polytoperealization} For a simplicial complex $\Delta$ with $n$ vertices, the following are equivalent:
\begin{enumerate}
\item $\Delta$ is $d$-convex union representable.
\item There exists a $d$-convex union representation $\{V_i\}_{i=1}^n$ of $\Delta$ such that (i) the collection $\{\ol{V_i}\}_{i=1}^n$ of closures of the $V_i$ has nerve $\Delta$,  (ii) each $\ol{V_i}$ is a polytope, and (iii) the union of all $V_i$ is the interior of a polytope. 
\item $\Delta$ is the nerve of a collection $\{P_i\}_{i=1}^n$ of $d$-dimensional polytopes whose union is a polytope in $\R^d$.
\item $\Delta$ is the nerve of a collection $\{A_i\}_{i=1}^n$ of convex closed sets whose union is a closed convex set in $\R^d$.
\end{enumerate}\end{proposition}
\begin{proof}
We first show that (1) implies (2). Let $\{U_i\}_{i=1}^n$ be a convex union representation of $\Delta$. By intersecting with an open ball of a sufficiently large radius, we can assume that all sets $U_i$ are bounded. Choose a representation $\{V_i\}_{i=1}^n$ as guaranteed by Lemma \ref{lem:shrink}. Properties (ii) and (iii) of (2) follow from the statement of Lemma \ref{lem:shrink}, so we just need to check that the nerve of $\{\overline V_i\}_{i=1}^n$ is $\Delta$. This is immediate from the fact that $V_i\subseteq\overline V_i\subseteq U_i$ for all $i\in n$.

The implications $(2)\Rightarrow (3)$ and $(3)\Rightarrow (4)$ are straightforward: the former by taking closures of the $V_i$, and the latter since polytopes are closed and convex. 

For the implication (4) $\Rightarrow$ (1), assume that all $A_i$ are compact by intersecting with a closed ball of sufficiently large radius. Then if $A_\sigma$ and $A_\tau$ are disjoint, they have positive distance, (i.e., $\min\{\|x-y\| \mid x\in A_\sigma, y\in A_\tau\}>0$), 
and so we can take the Minkowski sum of all $A_i$ with an open ball $B_\varepsilon$ of sufficiently small radius while preserving the nerve. The resulting collection $\{A_i+B_\varepsilon\}_{i=1}^n$ is a convex union representation of $\Delta$ in $\R^d$: the sets $A_i+B_\varepsilon$ are convex open sets, and so is their union $\bigcup_{i=1}^n \big(A_i+B_\varepsilon\big)=\big(\bigcup_{i=1}^n A_i\big)+B_\varepsilon$.
\end{proof}

The proof of Theorem \ref{thm:halfspace} relies on an application of Lemma \ref{lem:shrink}, paired with the methods of \cite{wegner} and the following lemma.

\begin{lemma}\label{lem:twofaces}
Let $\Gamma\subsetneq \Delta$ be acyclic simplicial complexes. Then $(\Delta,\Gamma)$ contains at least two faces. 
\end{lemma}
\begin{proof} Since $\{\emptyset\}$ is not acyclic, $(\Delta,\Gamma)$ contains at least one nonempty face.
Suppose for contradiction that $(\Delta,\Gamma)$ contains a unique face $\sigma$, and let $k=\dim\sigma\geq 0$. Observe that $\dc\sigma\cap \Gamma = \partial\dc\sigma$, and that $\Delta = \Gamma \cup \dc\sigma$. We obtain a Mayer-Vietoris sequence \[
\cdots\to \tilde{H}_{k}(\Delta)\to \tilde{H}_{k-1}(\partial\dc\sigma) \to \tilde{H}_{k-1}(\Gamma)\oplus \tilde{H}_{k-1}(\dc\sigma) \to  \cdots.
\]
Since $\Delta$, $\Gamma$, and $\dc\sigma$ are all acyclic this part of the sequence becomes $0 \to \tilde{H}_{k-1}(\partial\dc\sigma) \to 0$.
This is a contradiction to the fact that the boundary of $\dc\sigma$ has nonzero $(k-1)$-homology. 
\end{proof}

If $A$ is a(n oriented) hyperplane in $\R^d$, then $\R^d\setminus A$ consists of two open halfspaces which we denote by $A^+$ and $A^-$ --- the positive and negative side of $A$, respectively.
\begin{proof}[Proof of Theorem \ref{thm:halfspace}] 
Observe that it is enough to prove the result in the case that all $U_i$ are bounded. It also suffices to consider the case that $H$ is an open halfspace, since for a closed halfspace the nerve is unaffected by replacing the halfspace with its interior. Thus throughout the proof we assume that all $U_i$ are bounded and that $H$ is open. 

We work by induction on the number of faces in $\Delta$. There is nothing to prove if $\Delta=\emptyset$. If $\Delta$ is a single vertex, then every convex union representation consists of a single convex open set, and the nerve $\nerve(\{U_1\cap H\})$ is either $\Delta$ or the void complex, depending on whether $U_1\cap H$ is empty. In either case $\Delta$ collapses to $\nerve(\{U_1\cap H\})$ and the result follows.

Otherwise $\Delta$ has more than two faces. Again if $\nerve(\{U_i\cap H\}_{i=1}^n) = \Delta$ we are done. If not, let $A$ be the hyperplane defining the halfspace $H$, oriented so that $H = A^+$. Choose $\varepsilon$ so that every nonempty $U_\sigma\cap H$ contains an $\varepsilon$-ball, and apply Lemma \ref{lem:shrink} to obtain a new convex union representation $\{V_i\}_{i=1}^n$ of $\Delta$. Observe that by choice of $\varepsilon$ and property (iii) of Lemma \ref{lem:shrink}, $\nerve(\{U_i\cap H\}_{i=1}^n) = \nerve(\{V_i\cap H\}_{i=1}^n)$. Property (i) and the boundedness of the $U_i$ imply that if $V_\sigma\cap H = \emptyset$, then $V_\sigma$ and $H$ have positive distance to one another. Thus we may perturb the position and angle of $H$ slightly without changing the nerve $\nerve(\{V_i\cap H\}_{i=1}^n)$.

Perform a perturbation of $H$ so that it is in general position relative to $\{V_i\}_{i=1}^n$ in the following sense: no hyperplane parallel to $A$ simultaneously supports two disjoint nonempty $V_\sigma$ and $V_\tau$. For all facets $\sigma$ in the relative complex $\big(\Delta,\nerve(\{V_i\cap H\}_{i=1}^n)\big)$, let $d_\sigma$ be the distance from $V_\sigma$ to $H$. There is at least one such facet $\sigma$ since we are assuming that $\nerve(\{V_i\cap H\}_{i=1}^n)$ is a proper subcomplex of $\Delta$. The distances $d_\sigma$ are finite since each $V_\sigma$ is bounded, and they are distinct by genericity of $A$. Let $V_{\sigma_0}$ be the region whose distance to $H$ is largest. Let $A_0$ be the hyperplane separating $V_{\sigma_0}$ from $H$ and supporting $V_{\sigma_0}$, oriented so that $H$ lies on its positive side. Finally, let $\Gamma = \nerve(\{V_i\cap A_0^+\}_{i=1}^n)$. Then $\Gamma$ is a proper acyclic subcomplex of $\Delta$ containing $\nerve(\{V_i\cap H\}_{i=1}^n)$ and ${\sigma_0}$ is the unique facet of $(\Delta,\Gamma)$. Applying Lemma \ref{lem:twofaces} we conclude  that there is a minimal face $\tau_0\in \Delta\setminus \Gamma$ with $\tau_0\subsetneq \sigma_0$. Uniqueness of $\sigma_0$ implies that $\tau_0$ is a free face of $\Delta$.

We modify our representation one last time. By property (i) of Lemma \ref{lem:shrink}, disjoint $V_\sigma$ and $V_\tau$ have nonzero distance between them, and furthermore any $V_\sigma$ with $V_\sigma\cap H = \emptyset$ has nonzero distance to $H$. Let $\delta>0$ be smaller than one half the minimum of all these distances, and let $B_\delta$ be the open ball with radius $\delta$ centered at the origin. For $i\in[n]$, define \[
W_i = \begin{cases} V_i& i\in \tau_0\\
(V_i+B_\delta)\cap \bigcup_{i=1}^n V_i & i\notin \tau_0.\end{cases}
\]
By choice of $\delta$ the nerve of $\{W_i\}_{i=1}^n$ is equal to $\Delta$. Moreover the union of the $W_i$ is the same as the union of the $V_i$. Since $W_i = V_i$ for $i\in \tau_0$, it follows that $V_{\tau_0} = W_{\tau_0}$ and that $W_{\sigma_0}$ is supported by $A_0$. Finally, by choice of $\delta$, $\nerve(\{W_i\cap H\}_{i=1}^n) = \nerve(\{V_i\cap H\}_{i=1}^n)$.

\[
\includegraphics[width=30em]{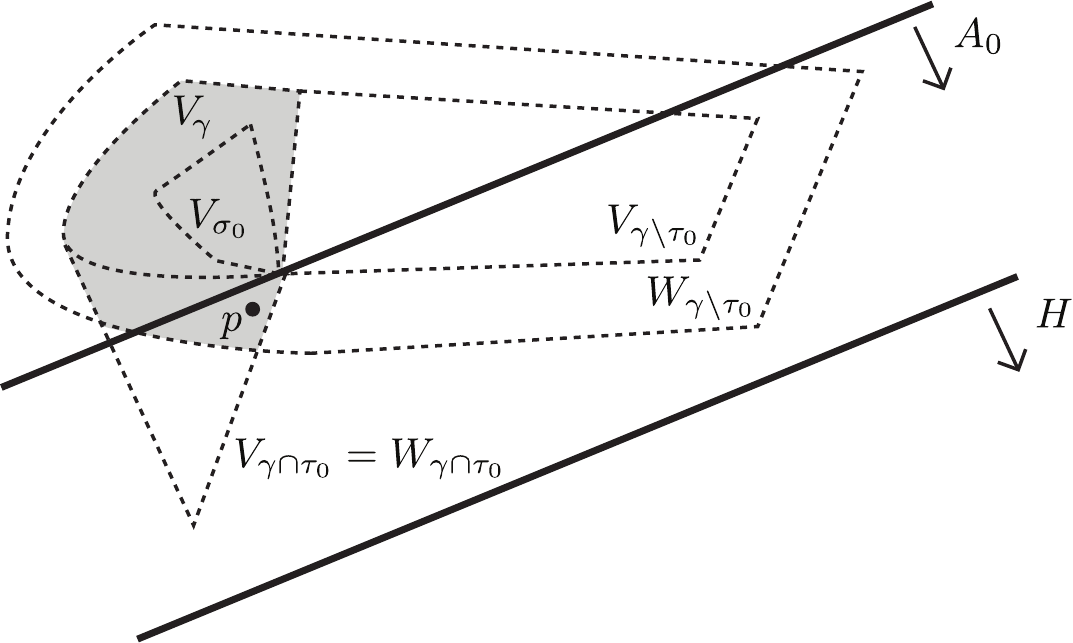}\]
Now for $i\in[n]$ define $X_i = W_i\cap A_0^+$. Then $\nerve(\{X_i\cap H\}_{i=1}^n) = \nerve(\{W_i\cap H\}_{i=1}^n)$ and $X_{\sigma_0} = W_{\sigma_0}\cap A_0^+ = \emptyset$. By inductive hypothesis, the nerve $\nerve(\{X_i\}_{i=1}^n)$ collapses to $\nerve(\{X_i\cap H\}_{i=1}^n)$. We claim that $\nerve(\{X_i\}_{i=1}^n)$ is equal to $\Delta\setminus \tau_0$. It suffices to show that $X_\gamma=W_\gamma\cap A_0^+$ is not empty for every $\gamma\in (\Delta, \Gamma)$ with $\tau_0\not\subseteq \gamma$. Note that for such a $\gamma$, $\tau_0\cap \gamma \subsetneq \tau_0$, and so by minimality of $\tau_0$, $W_{\tau_0\cap \gamma}\cap A_0^+=V_{\tau_0\cap \gamma}\cap A_0^+ \neq \emptyset$. 
Since $V_{\sigma_0}\subseteq W_{\sigma_0} \subseteq W_{\tau_0\cap \gamma}$ and $V_{\sigma_0}$ is supported by $A_0$, we may choose a point $p\in W_{\tau_0\cap \gamma}\cap A_0^+ $ which is arbitrarily close to $V_{\sigma_0}$. By construction of $W_i$ the region $W_{\gamma\setminus \tau_0}$ contains the Minkowski sum $V_{\sigma_0} + B_\delta$. But then it must contain $p$, so $W_\gamma = W_{\gamma\cap \tau_0} \cap W_{\gamma\setminus \tau_0}$ contains $p$. In particular, $W_\gamma$ has nonempty intersection with $A_0^+$. This situation is illustrated in the figure above. In the figure the shaded area is equal to $W_{\gamma}$. 

Since $W_\gamma\cap A_0^+ \neq \emptyset$ for all $\gamma\in \Delta\setminus \Gamma$ with $\tau_0\not\subseteq \gamma$, we conclude that $\nerve(\{X_i\}_{i=1}^n) = \Delta\setminus \tau_0$. Furthermore, since $\Delta\to \Delta\setminus \tau_0$ is a collapse, we conclude that $\Delta$ collapses to $\nerve(\{X_i\cap H\}_{i=1}^n) = \nerve(\{U_i\cap H\}_{i=1}^n)$, proving the result.
\end{proof}

%

\section{Constructible-like behavior} \label{sec:construtcable-like}
As we saw in Corollary \ref{cor:collapsiblebutnotrepresentable}, not all collapsible complexes are convex union representable. 
Thus our goal is to establish some additional necessary conditions for convex union representability. The corollaries enumerated in Section \ref{sec:collapsiblenotCUR} provide some such conditions. The following theorem is another step in this direction, and shows that convex union representable complexes are similar in spirit to constructible complexes --- a notion introduced in \cite{zeeman}. 

\begin{theorem}\label{thm:splitting}
Let $\Delta$ be a $d$-convex union representable simplicial complex, and let $\tau_1,\tau_2\in \Delta$ be such that $\tau_1\cup\tau_2 \notin\Delta$. Then there exist simplicial complexes $\Delta_1\subseteq \Delta\setminus \tau_1$ and $\Delta_2\subseteq \Delta\setminus \tau_2$ satisfying
\begin{itemize}
\item[(i)] $\Delta = \Delta_1\cup \Delta_2$;
\item[(ii)] $\Delta$ collapses onto $\Delta_i$ (for $i=1,2$);
\item[(iii)] $\Delta_i$ (for $i=1,2$) collapses onto $\Delta_1\cap \Delta_2$;
\item[(iv)] $\Delta_1$ and $\Delta_2$ are $d$-convex union representable, and
\item[(v)] $\Delta_1\cap \Delta_2$ is $(d-1)$-convex union representable.
\end{itemize}
\end{theorem}
\begin{proof}
Let $\{U_1,\ldots, U_n\}$ be a convex union representation of $\Delta$. The condition that $\tau_1\cup\tau_2 \notin \Delta$ implies that $U_{\tau_1}$ and $U_{\tau_2}$ are disjoint. Thus we may choose a hyperplane $H$ separating $U_{\tau_1}$ and $U_{\tau_2}$, oriented such that $U_{\tau_2}$ lies on the open positive side $H^+$ of $H$. Apply Lemma \ref{lem:shrink} to obtain a new representation $\{V_1,\ldots, V_n\}$ of $\Delta$. This representation has the property that if $\sigma\in \Delta$ and $V_\sigma\cap H = \emptyset$, then there is a positive distance between $V_\sigma$ and $H$. In particular, there is a small $\varepsilon$ so that the Minkowski sum of $H$ with an $\varepsilon$-ball induces the same nerve as $H$ when intersected with the various $V_i$.

Now, let $\Delta_1$ be the nerve of $\{V_1\cap H^+, \ldots, V_n\cap H^+\}$ and let $\Delta_2$ be the nerve of $\{V_1\cap H^-, \ldots, V_n\cap H^-\}$. We claim that $\Delta_1$ and $\Delta_2$ satisfy the conditions stated above. 

First let us argue that $\Delta_1\subseteq \Delta\setminus\tau_1$ and $\Delta_2\subseteq \Delta\setminus\tau_2$. Note that $\Delta_1\subseteq \Delta$ since the sets representing $\Delta_1$ are subsets of the sets representing $\Delta$. Moreover $\tau_1\notin \Delta_1$ since $V_{\tau_1}\cap H^+ = \emptyset$, and thus $\Delta_1\subseteq \Delta\setminus\tau_1$. A symmetric argument shows that $\Delta_2\subseteq \Delta\setminus\tau_2$. 

For (i), let $\sigma\in \Delta$. Then $V_\sigma\neq\emptyset$, and since $V_\sigma$ is open it has nonempty intersection with $H^+$ or with $H^-$. In the former case $\sigma\in \Delta_1$ and in the latter $\sigma\in \Delta_2$. Thus $\Delta = \Delta_1\cup\Delta_2$. 
For (ii) we can apply Theorem \ref{thm:halfspace} with the open halfspaces $H^+$ and $H^-$. 

To prove (iii), we first claim that $\Delta_1\cap \Delta_2 = \nerve(\{V_i\cap H\}_{i=1}^n)$. If $\sigma\in \Delta_1\cap \Delta_2$, then $V_\sigma$ contains points on both sides of $H$, and by convexity it contains points in $H$. Thus $\sigma\in \nerve(\{V_i\cap H\}_{i=1}^n)$. Conversely, if $\sigma\in \nerve(\{V_i\cap H\}_{i=1}^n) $, then $V_\sigma$ contains points in $H$, and by openness it contains points in both $H^+$ and $H^-$. 

To see that $\Delta_i$ collapses to $\nerve(\{V_i\cap H\}_{i=1}^n)$, let $C$ be the Minkowski sum of $H$ with a small $\varepsilon$-ball, so that $\nerve(\{V_i\cap H\}_{i=1}^n) = \nerve(\{V_i\cap C\}_{i=1}^n)$. Then observe that $C\cap H^+$ induces the nerve $\Delta_1\cap \Delta_2$ when intersected with the convex union representation $\{V_i\cap H^+\}_{i=1}^n$ of $\Delta_1$. By Corollary \ref{cor:collapsetoconvexset} this implies that $\Delta_1$ collapses onto $\Delta_1\cap\Delta_2$. A symmetric argument shows that $\Delta_2$ collapses onto $\Delta_1\cap \Delta_2$. 

For (iv) simply observe that $\{V_1\cap H^+, \ldots, V_n\cap H^+\}$ and $\{V_1\cap H^-, \ldots, V_n\cap H^-\}$ are $d$-convex union representations for $\Delta_1$ and $\Delta_2$ respectively. 

To prove (v) recall that $\Delta_1\cap \Delta_2 = \nerve(\{V_i\cap H\}_{i=1}^n)$. Since $H\cong \R^{d-1}$, this yields a $(d-1)$-convex union representation of $\Delta_1\cap\Delta_2$.
\end{proof}


\begin{corollary} \label{cor:suspension}
Let $\Delta$ be a simplicial complex that is not $d$-convex union representable. Then the suspension of $\Delta$ is not $(d+1)$-convex union representable. In particular, if $\Delta$ is not convex union representable, then neither is $\Sigma\Delta$.
\end{corollary}
\begin{proof}
Assume $\Sigma\Delta$ is $(d+1)$-convex union representable, and let $u$ and $v$ be the suspension vertices. Observe that $\{u,v\}$ is not a face of $\Sigma\Delta$, so by Theorem \ref{thm:splitting} there exist complexes $\Delta_1\subseteq \Sigma\Delta\setminus u=v\ast \Delta$ and $\Delta_2\subseteq \Sigma\Delta\setminus v=u\ast\Delta$ satisfying (i)-(v) in the theorem statement. But since $\Delta_1\cup \Delta_2 = \Sigma\Delta=v\ast\Delta\cup u\ast\Delta$, it must be the case that $\Delta_1 = v\ast \Delta$ and $\Delta_2 = u\ast \Delta$. Then $\Delta_1\cap \Delta_2 = \Delta$, and by (v) we conclude that $\Delta$ is $d$-convex union representable, a contradiction. 
\end{proof}

Note that Corollary \ref{cor:suspension} together with Corollary \ref{cor:collapsiblebutnotrepresentable} provides us with additional examples of collapsible complexes that are not convex union representable. In some situations, see Corolary \ref{cor:bound-dim} below, it also allows us to establish lower bounds on the minimum dimension of a convex union representation.

\section{Alexander duality} \label{sec:Alexanderdual}
Recall that if $\Delta$ is a simplicial complex with vertex set $[n]$, then the \emph{Alexander dual} of $\Delta$ is \[\Delta^* := \{\sigma\subseteq [n] \mid [n]\setminus\sigma\notin \Delta\}.\] The goal of this section is to show that if $\Delta$ is convex union representable and $n\geq 1$, then the Alexander dual of $\Delta$ is collapsible. It is worth mentioning that there exist collapsible complexes whose Alexander dual is not collapsible (see, for instance \cite[Example 3.3]{welker99}) --- such complexes thus provide additional examples of collapsible complexes that are not convex union representable. On the other hand, a complex is non-evasive if and only if its Alexander dual is non-evasive (see \cite{topologicalevasiveness} and \cite[Lemma 2.5]{welker99}). This suggests that convex union representability may imply non-evasiveness.

For our result we require the following standard lemma, which is proven in \cite{topologicalevasiveness}. 

\begin{lemma}\label{lem:dualcollapsible}
Let $\Delta\subseteq\Gamma$ be simplicial complexes. Then $\Gamma$ collapses onto $\Delta$ if and only if $\Delta^*$ collapses onto $\Gamma^*$. 
\end{lemma}
\begin{corollary}\label{cor:dualcollapsible}
Let $\Delta$ be a simplicial complex with vertex set $[n]$. Then $\Delta^*$ is collapsible if and only if $2^{[n]}$ collapses onto $\Delta$.
\end{corollary}
\begin{proof}
Take $\Gamma = 2^{[n]}$ and use Lemma \ref{lem:dualcollapsible}, noting that $\Gamma^* = \emptyset$. 
\end{proof}

With Corollary \ref{cor:dualcollapsible} in hand, we are ready to prove the main result of this section:

\begin{theorem} \label{thm:dualcollapsible}
Let $\Delta$ be a convex union representable complex with $n\geq 1$ vertices. Then $\Delta^*$ is collapsible.
\end{theorem}
\begin{proof}
Let $\{U_i\}_{i=1}^n$ be a convex union representation of $\Delta$ in $\R^d$ such that the collection of closures $\{\overline{U_i}\}_{i=1}^n$ has nerve equal to $\Delta$, as guaranteed by Proposition \ref{prop:polytoperealization}. Embed $\{U_i\}$ into $\R^{d+1}$ on the hyperplane $x_{d+1} = 0$. Let $U = \bigcup_{i\in[n]} U_i$, and let $V$ be the shifted copy of $U$ contained in the hyperplane $x_{d+1} = 1$. Then for $i\in[n]$ define $V_i = \interior(\conv(U_i\cup V))$. This construction is illustrated below.
\[
\includegraphics[scale=0.60]{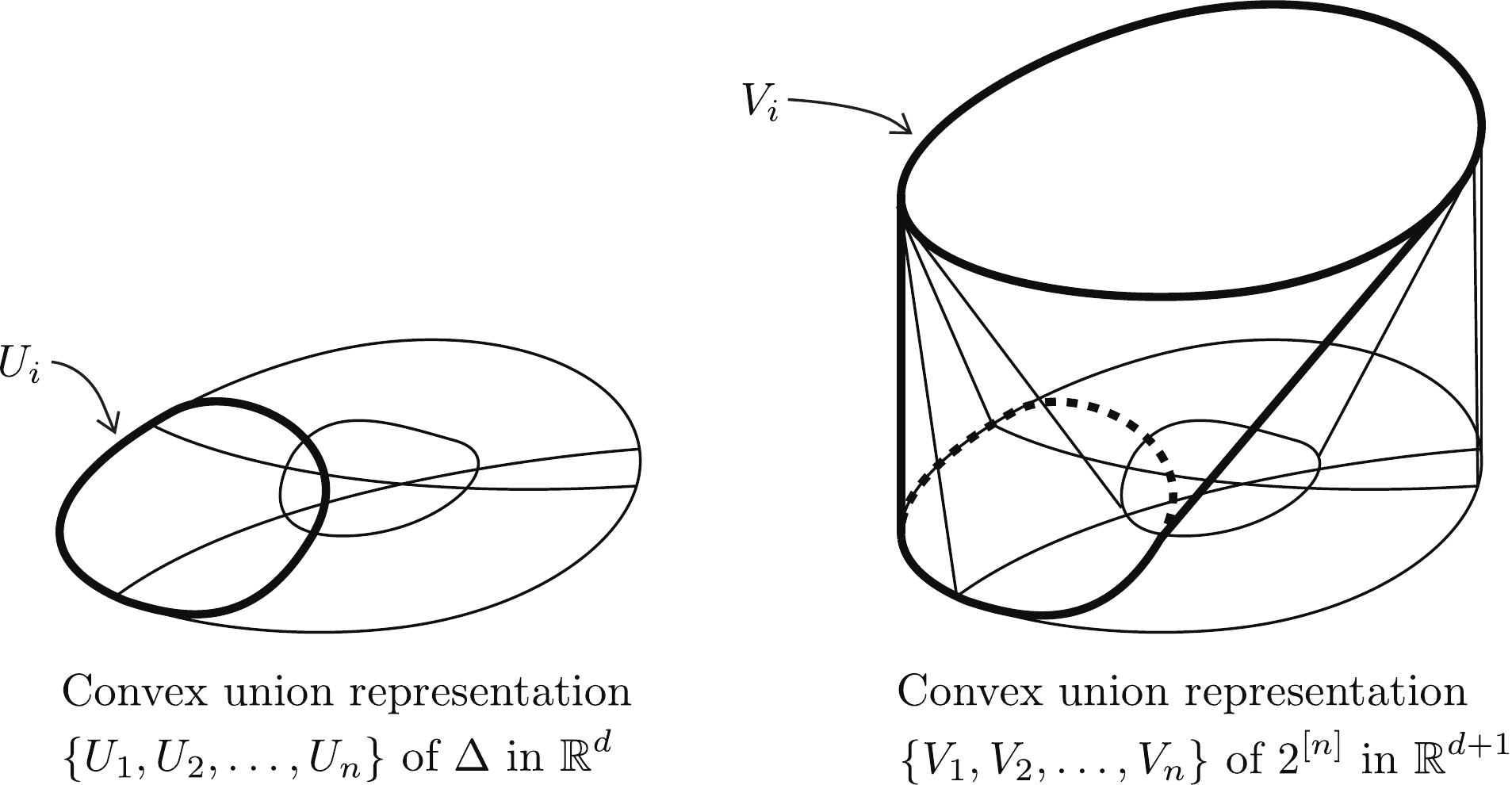}
\]

Observe that the nerve of $\{V_i\}_{i=1}^n$ is $2^{[n]}$. Since the $U_i$ were chosen such that their closures have the same nerve, it follows that for a sufficiently small $\varepsilon>0$, the halfspace $H_\varepsilon = \{(x_1,\ldots,x_d,x_{d+1})\in \R^{d+1}\mid x_{d+1}<\varepsilon\}$ has the property that  \[
\nerve(\{U_i\}) = \nerve(\{V_i\cap H_\varepsilon\}). 
\]
But by Theorem \ref{thm:halfspace}, this implies that $2^{[n]}$ collapses to $\Delta$. Corollary \ref{cor:dualcollapsible} then yields that $\Delta^*$ is collapsible. 
\end{proof}

\section{Convex union representable complexes with a few free faces} \label{sec:fewfreefaces}
In this section we bound the minimum dimension of a representation of a convex union representable complex $\Delta$ by the number of free faces of $\Delta$. More specifically, we establish the following result.

\begin{theorem} \label{thm:(k-1)rep}
If $\Delta$ is a convex union representable complex with $k$ facets that contain all free faces of $\Delta$, then $\Delta$ is $(k-1)$-convex union representable.
\end{theorem}
\begin{proof}
Let $\{U_i\}_{i=1}^n$ be a convex union representation of $\Delta$, and let $\sigma_1,\ldots, \sigma_k$ be the facets of $\Delta$ containing all free faces. Choose points $p_i\in U_{\sigma_i}$, and let $C = \conv(\{p_1,\ldots, p_k\})$. By Corollary \ref{cor:collapsetoconvexset}, $\Delta$ collapses to $\nerve(\{U_i\cap C\})$, but by choice of the $p_i$ the nerve $\nerve(\{U_i\cap C\})$ contains all free faces of $\Delta$. Thus $\Delta = \nerve(\{U_i\cap C\})$. Since $C$ is the convex hull of $k$ points, it is contained in an affine subspace of dimension no larger than $k-1$. Taking relative interiors in this affine subspace yields a convex union representation of $\Delta$ in dimension no larger than $k-1$, proving the result. 
\end{proof}

Two immediate consequences of Theorem \ref{thm:(k-1)rep} are

\begin{corollary}\label{cor:leray}
If $\Delta$ is a convex union representable complex with $k$ facets that contain all free faces of $\Delta$, then $\Delta$ is $(k-1)$-representable. In particular it is $(k-1)$-Leray. 
\end{corollary}

\begin{corollary}
Let $\Delta$ be a $d$-dimensional collapsible complex. Suppose that $\Delta$ has only $d$ or fewer free faces. Let $\Delta'$ be a stellar subdivision of $\Delta$ at one of its $d$-dimensional faces. Then $\Delta'$ is collapsible but not convex union representable. 
\end{corollary}
\begin{proof} Note that a stellar subdivision at a facet does not affect collapsibility nor the collection of free faces.
Let $\sigma\in\Delta$ be the $d$-face at which the subdivision occurs. Then $\Delta'$ will have the boundary of $\dc\sigma$ as an induced subcomplex. This boundary has nonvanishing $(d-1)$-homology, and so $\Delta'$ is not $(d-1)$-Leray. On the other hand, $\Delta'$ has only $d$ or fewer free faces. Thus by Corollary \ref{cor:leray}, $\Delta'$ is not convex union representable. 
\end{proof}

The following corollary of Theorem \ref{thm:(k-1)rep} characterizes convex union representable complexes that have at most two free faces; in particular, it provides a different proof of Corollary~\ref{cor:collapsiblebutnotrepresentable}.  

\begin{corollary}\label{cor:threefreefaces}
Let $\Delta$ be a convex union representable complex. Then $\Delta$ has at most two free faces if and only if $\Delta$ is a path. 
\end{corollary}
\begin{proof} 
One direction is clear: paths are $1$-convex union representable complexes that have at most two free faces. For the other direction,
assume that $\Delta$ is a convex union representable complex with at most two free faces. Then by Corollary \ref{cor:leray}, $\Delta$ is $1$-representable. Thus to prove the result, it suffices to show that if $\Delta$ is a collapsible, $1$-representable complex with at most two free faces, then $\Delta$ is a path.

The class of $1$-representable complexes, also known as clique complexes of interval graphs,  is well understood. In particular, \cite[Theorem 7.1]{intervalgraphs} asserts that $\Gamma$ is $1$-representable if and only if $\Gamma$ satisfies the following property (*): the facets of $\Gamma$ can be numbered $F_1,\ldots, F_m$ in such a way that for all pairs $(i,k)$ with $1\leq i <k\leq m$ and for any vertex $v\in\Gamma$, if $v\in F_i \cap F_k$, then $v$ belongs to \emph{all} $F_j$ with $i<j<k$.  

We use induction on $m$ to show that any collapsible $\Gamma$ that satisfies (*) and has at most two free faces must be a path.
In the base case of $m=1$, $\Gamma$ is a simplex $\overline{F_1}$, in which case the result is immediate: indeed, $0$- and $1$-dimensional simplices form a path, while a simplex of dimension $d\geq 2$ has more than two free faces. 

For the case of $m>1$, let $v$ be a vertex that belongs to $F_m$, but not to $F_{m-1}$. (It exists since $F_m$ is a facet, and hence it is not a subset of $F_{m-1}$.) Similarly, let $w\in F_1\setminus F_2$. Since $\Gamma$ satisfies (*), $F_m$ is the only facet that contains $v$. Consequently, $v\neq w$, and any proper subset of $F_m$ that contains $v$ is a free face of $\Gamma$. By a symmetric argument, $w$ is a free face of $\Gamma$. Thus, for $\Gamma$ to have at most two free faces, $F_m$ must be $1$-dimensional. We write $F_m=\{v,u\}$ and consider $\Gamma'=\Gamma\setminus v=\bigcup_{i=1}^{m-1}\overline{F_i}$. Then $\Gamma'$ is collapsible and satisfies property (*). Can it happen that $\Gamma'$ has more than two free faces? If yes, then at least two of these free faces are not $\{u\}$; hence they are also free faces of $\Gamma$, which together with the face $\{v\}$ already accounts for three free faces of $\Gamma$, contradicting our assumption. Thus $\Gamma'$ has at most two free faces, and hence $\Gamma'$ is a path by inductive hypothesis. We conclude that $\Gamma=\Gamma' \cup \overline{\{u,v\}}$ is a $1$-dimensional collapsible complex, so a tree.  The result follows since every tree that is not a path has at least three leaves, and each leaf is a free face. 
\end{proof}

\section{Constructions} \label{sec:constructions}

In this section we discuss some sufficient conditions for convex union representability. We start with the cones and joins.

\begin{remark}
Every simplicial complex $\Delta$ with vertex set $[n]$ is the nerve of a collection $\{U_i\}_{i=1}^n$ of convex open sets (whose union is not necessarily convex). Indeed, every simplicial complex has a representation $\{A_i\}_{i=1}^n$ consisting of compact convex sets (in fact polytopes, see \cite[Section 3.1]{tancer}). Compactness implies that disjoint regions $A_\sigma$ and $A_\tau$ have positive distance between them, so we can replace each $A_i$ by its Minkowski sum with a small open ball without affecting the nerve. 
\end{remark}

\begin{proposition}\label{prop:cone}
Every cone is convex union representable. 
\end{proposition}
\begin{proof}
Consider a cone $v\ast\Delta$. Choose a $d$-representation of $\Delta$ by convex open sets. Adding $U_v = \R^d$ to this $d$-representation yields a $d$-convex union representation of $v\ast\Delta$.
\end{proof}

\begin{remark}
Proposition \ref{prop:cone} and the fact that all convex union representable complexes are collapsible and hence $\mathbb{Q}$-acyclic leads to the following set of (strict) implications:
\[
\mbox{ cone } \Longrightarrow \mbox{ convex union representable } \Longrightarrow \mbox{ $\mathbb{Q}$-acyclic}.
\]
Since the set of $f$-vectors of the class of simplicial complexes that are cones coincides with set of $f$-vectors of the class of $\mathbb{Q}$-acyclic simplicial complexes \cite{f-acyclic}, it follows that it also coincides with the set of $f$-vectors of the class of convex-union representable complexes. 
\end{remark}

\begin{proposition}\label{prop:join}
Let $\Delta$ be a $d_1$-convex union representable complex, and let $\Gamma$ be a $d_2$-representable simplicial complex. Then $\Delta\ast \Gamma$ is $(d_1+d_2)$-convex union representable.
\end{proposition}
\begin{proof}
Let $A$ and $B$ be the (disjoint) vertex sets of $\Delta$ and $\Gamma$ respectively. Let $\{U_i\}_{i\in A}$ be a $d_1$-convex union representation of $\Delta$, and let $U = \bigcup U_i$. Let $\{V_j\}_{j\in B}$ be a $d_2$-representation of $\Gamma$ consisting of convex open sets. Define $V = \conv\left(\bigcup V_j \right)$. Then for $k\in A\cup B$ define \[
W_k = \begin{cases}
U_k\times V & k\in A,\\
U\times V_k & k\in B.
\end{cases}
\]
The union of all $W_k$ is the convex open set $U\times V$ since $U$ is the union of all $U_i$. Moreover the nonempty regions of intersection among the $W_k$ are of the form $U_\sigma \times V_\tau$ for $\sigma\in \Delta$ and $\tau\in \Gamma$, and so $\nerve(\{W_k\}_{k\in A\cup B}) = \Delta\ast \Gamma$. Thus $\Delta\ast\Gamma$ is $(d_1+d_2)$-convex union representable as desired. 
\end{proof}

In some instances, Proposition \ref{prop:join} along with Corollary \ref{cor:suspension} allows us to compute the minimum dimension of a convex union representation of a given complex. Below is one such example.

\begin{corollary} \label{cor:bound-dim}
 Let $\Delta=\ol{p}$ be a point, and let $\Sigma^k\Delta$ denote the $k$-fold suspension of $\Delta$. Then for all $k\ge 0$, the complex $\Sigma^k\Delta$ is $k$-convex union representable but not $(k-1)$-convex union representable. 
\end{corollary}
\begin{proof}
We work by induction on $k$. If $k=0$ then $\Sigma^k\Delta= \Delta$ and the result holds. Suppose $k\ge 1$ and let $\Gamma$ be  the complex consisting of two points. Observe that $\Gamma$ is 1-representable. Further observe that $\Sigma^k \Delta = (\Sigma^{k-1}\Delta)\ast \Gamma$, and so Proposition \ref{prop:join} implies that $\Sigma^k\Delta$ is $k$-convex union representable. Since $\Sigma^{k-1}\Delta$ is not $(k-2)$ convex union representable, Corollary \ref{cor:suspension} implies that $\Sigma^k\Delta$ is not $(k-1)$-convex union representable. The result follows. 
\end{proof}

The following proposition shows that building cones over certain subcomplexes of convex union representable complexes preserves convex union representability.

\begin{proposition}\label{prop:building}
Let $\Delta\subseteq 2^{[n]}$ be a $d$-convex union representable simplicial complex. Choose a face $\sigma\in\Delta$ and a set $\omega$ with $\sigma\subseteq \omega \subseteq [n]$. Then the complex \[
\Delta':=\Delta \cup \big((n+1) \ast\cstar_\Delta(\sigma)|_\omega\big)
\]
is $(d+1)$-convex union representable. 
\end{proposition}
\begin{proof}
Let $\{U_i\}_{i=1}^n$ be a convex union representation of $\Delta$ in $\R^d$ consisting of bounded sets, and let $U = \bigcup_{i=1}^n U_i$. Choose an open convex set $V\subseteq U_\sigma$ such that $\overline V\subseteq U_\sigma$, and $V\cap U_\tau\neq\emptyset$ for all $\tau\supset \sigma$. Identify $\R^d$ with the hyperplane defined by $x_{d+1} = 0$ in $\R^{d+1}$, and let $\widetilde V$ be the shifted copy of $V$ contained in the hyperplane defined by $x_{d+1} = 1$. Define $W = \interior\conv(U\cup \widetilde V)$, and for $i\in [n]$ define \[
W_i = (U_i\times (0,1))\cap W. 
\]
Observe that $\{W_i\}_{i=1}^n$ is a $(d+1)$-convex union representation of $\Delta$ and that the union of all $W_i$ is equal to $W$. 

For $\varepsilon >0$, consider the hyperplane $H_\varepsilon$ defined by $x_{d+1}= 1-\varepsilon$, oriented so that $\widetilde V$ lies on the positive side. We claim that for small enough $\varepsilon$, $H^+_\varepsilon \cap W \subseteq W_\sigma$. Suppose not, so that there exists some $W_\tau\not\subseteq W_\sigma$ with $W_\tau\setminus W_\sigma$ containing points arbitrarily close to the hyperplane defined by $x_{d+1} = 1$. Since the only such points in $W$ are those approaching $V\times (0,1)$, we see that $W_\tau$ must contain points arbitrarily close to $V\times (0,1)$ but not in $W_\sigma$. But then the projection of these points onto $\R^d$ yields a series of points in $U$ approaching $V$, but not contained in $U_\sigma$. This contradicts the condition that $\overline V\subseteq U_\sigma$. We conclude that for some $\varepsilon >0$, $H^+_\varepsilon \cap W \subseteq W_\sigma$.

Now, for $i\notin \omega$, replace $W_i$ with $W_i\cap H_\varepsilon^-$. By choice of $\varepsilon$ this does not affect the nerve or union of $\{W_i\}_{i=1}^n$. Then define $W_{n+1} = W\cap H_\varepsilon^+$. We claim that $\{W_i\}_{i=1}^{n+1}$ is a $(d+1)$-convex union representation of $ \Delta \cup \big((n+1) \ast\cstar_\Delta(\sigma)|_\omega\big)$. To see this, observe that the regions of maximal intersection in $W\cap H_\varepsilon^+$ are the inclusion-maximal faces $\gamma$ of $\Delta$ satisfying $\sigma\subseteq \gamma \subseteq \omega$. These are exactly the facets of $\cstar_\Delta(\sigma)|_\omega$.  Thus the new faces introduced by $W_{n+1}$ are of the form $\{n+1\}\cup \gamma$ for $\gamma\in \cstar_\Delta(\sigma)|_\omega$, and the result follows.
\end{proof}

An immediate consequence of Proposition \ref{prop:building} is that trees are convex union representable:

\begin{corollary} \label{cor:tree}
A 1-dimensional complex $\Delta$ is convex union representable if and only if it is collapsible (in particular, if and only if it is a tree). 
\end{corollary}
\begin{proof}
If $\Delta\subseteq 2^{[n]}$ is a tree, we can build a convex union representation inductively using Proposition \ref{prop:building}. In the base case that $\Delta$ is a single vertex, a realization is given by $U_1=\{0\}= \R^0$. For $\Delta$ containing at least one edge, label the vertices so that $n+1$ is a leaf, and let $i$ be  the vertex that $n+1$ is adjacent to. Then $\Delta|_{[n]}$ is convex union representable by inductive hypothesis. Choosing $\sigma = \omega = \{i\}$  and applying Proposition \ref{prop:building} we obtain that the complex \[
\Delta|_{[n]} \cup \dc{\{i, n+1\}}
\]
is convex union representable. But this is just $\Delta$, so the result follows. 
\end{proof}

\begin{remark} 
In fact, it is not hard to show by induction that all trees are $2$-convex union representable.  This expands  a result of \cite{sparse}, which showed that trees (and in fact all planar graphs) are 2-representable.
\end{remark}

\begin{remark}
Essentially the same argument as in Corollary \ref{cor:tree} shows that if $\Delta$ is \emph{strong collapsible} and flag, then $\Delta$ is convex union representable. Strong collapsible complexes were introduced in \cite[Section 2]{BarmakMinian}. The condition of being a flag complex is only needed to guarantee that all vertex links are induced subcomplexes. In particular, since barycentric subdivisions are always flag and since the barycentric subdivision of any strong collapsible complex is strong collapsible \cite[Theorem 4.15]{BarmakMinian}, it follows that barycentric subdivisions of strong collapsible complexes are convex union representable. For example, the barycentric subdivision of an arbitrary cone complex is convex union representable.
\end{remark}

A simplicial complex $\Delta$ is a \emph{simplicial ball} (or a \emph{simplicial sphere}) if the geometric realization of $\Delta$ is homeomorphic to a ball (a sphere, respectively). Another immediate application of Proposition \ref{prop:building} is that all stacked balls are convex union representable. (Stacked balls are defined recursively: a $d$-dimensional simplex is a stacked ball with $d+1$ vertices; a $d$-dimensional stacked ball $\Delta$ on $n+1\geq d+2$ vertices is obtained from a $d$-dimensional stacked ball $\Gamma$ on $n$ vertices by choosing a free ridge $\sigma$ of $\Gamma$ and building a cone on it: $\Delta=\Gamma\cup((n+1)\ast\ol{\sigma})$.)
We close this section by showing that all simplicial balls in a certain larger class 
are convex union representable. 
Recall that if $P$ is a simplicial polytope of dimension $d$ and $v$ is a vertex of $P$, then $\partial P$ and $\partial P\setminus v$ are simplicial complexes of dimension $d-1$: the former is a simplicial sphere while the latter is a simplicial ball.
 
\begin{proposition} \label{prop:regularball}
Let $P$ be a $d$-dimensional simplicial polytope, and let $v$ be a vertex of $P$. Then $\partial P\setminus v$ is a $(d-1)$-convex union representable simplicial complex.
\end{proposition}
\begin{proof}
Translate $P$ if necessary so that the origin is in the interior of $P$. Let $P^\ast$ be the polar polytope of $P$, and let $F=\hat{v}$ be the facet of $P^\ast$ corresponding to the vertex $v$. Consider the Schlegel diagram $\mathcal{S}(F)$ of $P^\ast$ based at $F$. This is a \emph{polytopal} complex; in particular the facets of $\mathcal{S}(F)$ are polytopes. 
Furthermore, $\mathcal{S}(F)$ satisfies the following properties (see Chapters 2, 5, and 8 of \cite{ziegler-polytopesbook} for basics on polar polytopes, Schlegel diagrams, and polytopal complexes): (i) the set of facets of $\mathcal{S}(F)$ is in bijection with the set of facets of $P^\ast$ other than $F$, which in turn is in bijection with the vertex set $\mathcal{V}$ of $\partial P\setminus v$, that is, we can index the facets of $\mathcal{S}(F)$ by vertices of $\partial P\setminus v$:  $\{G_w\}_{w\in \mathcal{V}}$; (ii) for $w_1,\ldots, w_k\in \mathcal{V}$, the facets $G_{w_1},\ldots, G_{w_k}$ have a nonempty intersection if and only if $\{w_1,\ldots, w_k\}$ is a face of $\partial P\setminus v$; (iii) the union of all facets $G_w$ of $\mathcal{S}(F)$ is $F$. These properties imply that $\{G_w\}$ is a collection of closed convex subsets of $\text{Aff}(F)\cong \R^{d-1}$ whose nerve is $\partial P\setminus v$, and whose union is convex. This along with Proposition \ref{prop:polytoperealization} yields the result.
\end{proof}

\begin{corollary} \label{cor:2-balls}
Let $\Delta$ be an arbitrary $2$-dimensional simplicial ball. Then $\Delta$ is $2$-convex union representable.
\end{corollary}
\begin{proof}
Let $v$ be a vertex not in $\Delta$, and let $\partial\Delta$ denote the boundary of $\Delta$, that is, $\partial\Delta$ is the $1$-dimensional subcomplex of $\Delta$ whose facets are precisely the free edges of $\Delta$. Let $\Lambda:=\Delta \cup (v\ast \partial\Delta)$. Then $\Lambda$ is a $2$-dimensional simplicial sphere, and so by Steinitz' theorem (see \cite[Chapter 4]{ziegler-polytopesbook}), $\Lambda$ can be realized as the boundary complex of a simplicial polytope. Since $\Delta=\Lambda\setminus v$, the previous proposition implies the result.
\end{proof}

The situation with higher-dimensional balls is much more complicated. For instance, there exist $3$-dimensional simplicial balls that are not even collapsible. (See \cite{knots-in-balls} for an explicit non-collapsible example with only $15$ vertices.) It would be interesting to understand which collapsible triangulations of balls are convex union representable.

\section{Nerve Obstructions to Convexity in Neural Codes}\label{sec:codes}
The goal of this section is to apply our results to the study of convex neural codes. Convex neural codes are a topic of recent research, see for example \cite{undecidability, chadvlad, local15, nogo, neuralring13, obstructions}. Informally, the theory of neural codes aims to answer the question ``what are the possible intersection patterns of collections of convex open sets?" We begin with some background to make this question more precise. 

A \emph{neural code} or \emph{combinatorial code} is a subset $\C$ of $2^{[n]}$. Given a collection $\{U_i\}_{i=1}^n$ of convex open sets in $\R^d$ the \emph{code of $\{U_i\}_{i=1}^n$} is the neural code \[
\code(\{U_i\}_{i=1}^n) \od \bigg\{\sigma\subseteq [n] \Bigm| U_\sigma\setminus \bigcup_{j\notin\sigma} U_j\neq \emptyset\bigg\}
\]
where $U_\emptyset=\conv(\bigcup_{i=1}^n U_i)$ as earlier in the paper. 
The collection $\{U_i\}_{i=1}^n$ is called a \emph{convex realization} of  $\code(\{U_i\}_{i=1}^n)$. If a code $\C$ has a convex realization, we call $\C$ a \emph{convex code}, and otherwise we say that $\C$ is non-convex. Indices in $[n]$ may be referred to as \emph{neurons}, and $U_i$ is the \emph{receptive field} of a neuron $i$.

\begin{remark}
The convention that $U_\emptyset=\conv(\bigcup_{i=1}^n U_i)$ differs somewhat from existing neural code literature, in which one usually specifies an ``ambient space" $X$ containing all $U_i$, and defines $U_\emptyset = X$. Our convention can be thought of as treating $\conv(\bigcup_{i=1}^n U_i)$ as an implicit ambient space. This convention is motivated by the fact that, with it,  $\emptyset \notin \code(\{U_i\}_{i=1}^n)$ if and only if $\bigcup_{i=1}^n U_i$ is convex. More generally (and independent of this convention) we have for nonempty $\sigma$ that $\sigma\notin \code(\{U_i\}_{i=1}^n)$ if and only if $\{U_j\mid j\notin \sigma\}$ covers $U_\sigma$.
\end{remark}

For any neural code $\C$, the smallest simplicial complex containing $\C$ is denoted $\Delta(\C)$. The topology of this simplicial complex has been used to characterize certain ``local obstructions" to convexity. In particular, \cite{local15, nogo} showed that if $\C$ is a convex code then for all $\sigma\in \Delta(\C)\setminus \C$ the link $\link_{\Delta(\C)}(\sigma)$ must be contractible. In \cite{undecidability} the authors strengthened ``contractible" to ``collapsible." A failure of $\link_{\Delta(\C)}(\sigma)$ to be contractible is called a \emph{local obstruction of the first kind at $\sigma$}, while a failure to be collapsible is a  \emph{local obstruction of the second kind}.
Based on similar arguments, we define ``nerve obstructions" and further strengthen ``collapsible" to ``convex union representable."

\begin{definition}
Let $\C\subseteq 2^{[n]}$ and $\sigma\in \Delta(\C)\setminus \C$. We say that $\C$ has a \emph{nerve obstruction} at $\sigma$ if $\link_{\Delta(\C)}(\sigma)$ is not convex union representable. If $\C$ has no nerve obstructions, then $\C$ is called \emph{locally perfect}.
\end{definition}

\begin{remark}
Nerve obstructions generalize local obstructions in the following sense. If $\C$ has a local obstruction of the first or second kind at $\sigma$, then $\C$ has a nerve obstruction at $\sigma$.
\end{remark}

\begin{proposition}\label{prop:nerveobstruction} Let $\C\subseteq 2^{[n]}$ be a neural code.
If $\C$ is convex then $\C$ is locally perfect. In particular, if $\C$ has a realization in $\R^d$ then for all $\sigma\in \Delta(\C)\setminus \C$ the link $\link_{\Delta(\C)}(\sigma)$ is $d$-convex union representable. 
\end{proposition}
\begin{proof}
Let $\{U_i\}_{i=1}^n$ be a realization of $\C$ in $\R^d$ as a neural code, and let $\sigma\in \Delta(\C)\setminus \C$. Consider the collection $\{U_j\cap U_\sigma \mid j\notin \sigma\}$. Since $\sigma\notin \C$ this collection covers $U_\sigma$. That is, the union of sets in this collection is the convex open set $U_\sigma$.  Moreover the nerve $\nerve(\{U_j\cap U_\sigma \mid j\notin \sigma\})$ is exactly $\link_{\Delta(\C)}(\sigma)$, so this collection gives a $d$-convex union representation of the link, proving the result. 
\end{proof}

Proposition \ref{prop:nerveobstruction} applied to $\sigma=\emptyset$ yields the following:

\begin{corollary}\label{cor:nerveobstruction}
Let $\Gamma$ be a simplicial complex that is not convex union representable. Then $\Gamma\setminus\{\emptyset\}$ is not a convex neural code. 
\end{corollary}
\begin{proof} Note that $\Delta(\Gamma\setminus\{\emptyset\})=\Gamma$ and $\link_\Gamma(\emptyset)=\Gamma$. Thus, if $\Gamma$ is not convex union representable, then the code $\Gamma\setminus\{\emptyset\}$ has a nerve obstruction at $\emptyset$, and so it is not convex.
\end{proof}

Corollary \ref{cor:nerveobstruction} gives us new families of examples of neural codes that are not convex. In particular, if $\Gamma$ is one of the collapsible but non-convex union representable complexes $\Sigma_d$ or $E_d$ of Corollary \ref{cor:collapsiblebutnotrepresentable}, then $\Gamma\setminus \{\emptyset\}$ is a non-convex neural code. 

Similarly to the case of usual local obstructions \cite{local15}, a
useful property of nerve obstructions is that they can only occur at intersections of maximal codewords. This makes searching for local obstructions more efficient. 
\begin{proposition}\label{prop:missingmaxint}
Let $\C\subseteq 2^{[n]}$ be a neural code with a nerve obstruction at $\sigma\subseteq [n]$. Then $\sigma$ is an intersection of maximal codewords of $\C$. 
\end{proposition}
\begin{proof}
We argue the contrapositive. Suppose that $\sigma$ is not an intersection of maximal codewords. Then $\link_{\Delta(\C)}(\sigma)$ is a cone, which is convex union representable by Proposition \ref{prop:cone}. Thus $\C$ does not have a nerve obstruction at $\sigma$. 
\end{proof}

We close this section with several remarks.
\begin{remark} It is important to note 
that the converse to Proposition \ref{prop:nerveobstruction} is false, that is, there exist non-convex locally perfect codes. A first example is given in \cite[Theorem 3.1]{obstructions}. The non-convex code on $5$ neurons described in this theorem has no local obstructions, and in fact it has no nerve obstructions. (The latter follows from a simple fact all contractible complexes with at most 4 vertices are convex union representable.) We build on this example in \cite{sunflowers}, providing an infinite family of locally perfect codes that are not convex. 
\end{remark}

\begin{remark} \label{rem:closed-vs-open}
Instead of considering codes arising from a collection of open convex sets in $\R^d$, one may consider codes arising from a collection of \emph{closed} convex sets. While according to \cite{chadvlad} the classes of codes realizable by closed convex sets and those realizable by open convex sets are distinct, and in fact neither class contains the other, the results in this paper (e.g., Proposition \ref{prop:nerveobstruction} and Corollary \ref{cor:nerveobstruction}) apply equally well for both types of codes: open-convex and closed-convex. This is because convex union representability does not change if we replace the openness requirement with closedness, see Proposition \ref{prop:polytoperealization}.
\end{remark}

\section{Concluding remarks and open problems}\label{sec:conclusion}

We have seen that convex union representability does not follow from familiar combinatorial properties such as collapsibility, shellability, or non-evasiveness. Observe also that being pure and convex union representable does not imply being shellable or even constructible. Indeed, let $\Gamma$ be a pure complex that is not constructible, and let $\Delta=v\ast\Gamma$. Then $\Gamma$ is also not constructible, however, $\Gamma$ is a cone, and so it is convex union representable by Proposition~\ref{prop:cone}.

It is also worth mentioning that the class of convex union representable complexes is {\bf not} closed under arbitrary collapses. To see this, let $\Delta'$ be a collapsible but not convex union representable complex, and let $\Delta=v\ast \Delta'$. Then $\Delta$ is  convex union representable by Proposition \ref{prop:cone}. Furthermore, $\Delta$ collapses to $\Delta'$, but $\Delta'$ is not convex union representable. (The fact that $\Delta$ collapses to $\Delta'$ follows from the following standard result: Let $\Gamma$ be a simplicial complex and let $\Gamma'\subseteq \Gamma$ be a nonempty collapsible subcomplex. If $v$ is a vertex not in $\Gamma$, then $(v \ast \Gamma')\cup\Gamma$ collapses onto $\Gamma$.)

In fact, recognizing $d$-convex union representable complexes is as hard as recognizing $d$-representable complexes. We thank Martin Tancer for bringing to our attention the following result:
\begin{proposition} \label{propos:NP-hard}
It is NP-hard to recognize $2$-convex union representable complexes.
\end{proposition}
\begin{proof} First note that a complex $\Delta$ is  $d$-representable if and only if the cone $v\ast\Delta$ is $d$-convex union representable. Indeed, the proof of Proposition~\ref{prop:cone} implies that if $\Delta$ is  $d$-representable, then $v\ast\Delta$ is $d$-convex union representable. Conversely, if $v\ast\Delta$ is $d$-convex union representable, then it is $d$-representable, and so $\Delta$ is $d$-representable: deleting from a $d$-representation of $v\ast\Delta$ the set that represents $v$ produces a $d$-representation of $\Delta$. The proposition then follows from a result of \cite{Tancer-NPcomplete} asserting that recognizing $2$-representable complexes is NP hard.
\end{proof}

While the above discussion indicates that the problem of characterizing the class of convex union representable complexes is out of reach at the moment, the following problems and questions partly motivated by our results might be more approachable.

\begin{question}
Is every convex union representable complex non-evasive? 
\end{question}

\begin{question}
Is the Alexander dual of a convex union representable complex always convex union representable?
\end{question}

\begin{question}
Is every shellable (or constructible) simplicial ball a convex union representable complex?
\end{question}

 \begin{question} \label{question:topological}
Can one characterize the class of topological spaces that possess convex union representable triangulations? Is it true that every collapsible complex becomes convex union representable after a sufficiently fine subdivision?
\end{question}

Several remarks related to Question \ref{question:topological} are in order. First, note that there exist contractible spaces, e.g., the dunce hat, that admit no collapsible triangulation; in particular, such a space has no triangulation that is convex union representable. Note also that the first barycentric subdivision of a collapsible complex is not always a convex union representable complex. Indeed, if $\Delta$ is a collapsible complex with only one free face (e.g., $\Sigma_d$ from Corollary \ref{cor:collapsiblebutnotrepresentable}) then it follows from Corollary \ref{cor:nocommonvertex} that the barycentric subdivision of $\Delta$ is \emph{not} convex union representable. It might still be the case that for a sufficiently large $n$, the $n$-th barycentric subdivision of a collapsible complex is convex union representable. 

\begin{question} \label{question:d-infty}
For a fixed $d\ge 2$, do there exist $d$-dimensional convex union representable complexes which require arbitrarily high dimension to represent?
\end{question}

\begin{question} \label{question:d-(d+1)}
Does there exist a simplicial complex $\Delta$ which is $d$-representable and $(d+1)$-convex union representable, but not $d$-convex union representable?
\end{question}

\noindent Since every $k$-dimensional simplicial complex is $(2k+1)$-representable \cite[Section 3.1]{tancer}, an affirmative answer to Question \ref{question:d-infty} would imply an affirmative answer to Question \ref{question:d-(d+1)}.

\section*{Acknowledgments} We are grateful to Bruno Benedetti, Florian Frick, Anne Shiu, and Martin Tancer for insightful conversations, comments, and questions on an earlier version of this paper.  Additional thanks go to the referees for helpful suggestions.

\bibliographystyle{plain}
\bibliography{neuralcodereferences}

\end{document}

%% file: neuralcodecommands.tex
\usepackage{amsmath}
\usepackage{amsfonts}
\usepackage{colordvi}
\usepackage{color}
\usepackage{amsthm}
\usepackage{url}
\usepackage{graphicx}
\usepackage{epstopdf}
\usepackage{amssymb}

\newtheorem{theorem}{Theorem}[section]
\newtheorem{lemma}[theorem]{Lemma}

\newtheorem{question}[theorem]{Question}
\newtheorem{proposition}[theorem]{Proposition}

\newtheorem{corollary}[theorem]{Corollary}
\theoremstyle{definition}

\newtheorem{definition}[theorem]{Definition}
\theoremstyle{remark}
\newtheorem{remark}[theorem]{Remark}

\newcommand{\C}{\mathcal C}

\newcommand{\U}{\mathcal U} 
\newcommand{\V}{\mathcal V}

\newcommand{\R}{\mathbb R}
\newcommand{\Z}{\mathbb Z}

\DeclareMathOperator{\conv}{conv}

\newcommand{\dc}[1]{\overline{#1}}

\DeclareMathOperator{\link}{Lk}

\DeclareMathOperator{\cstar}{St}

\DeclareMathOperator{\code}{code}
\DeclareMathOperator{\interior}{int}

\DeclareMathOperator{\dist}{dist}

\newcommand{\nerve}{\mathcal N}

\newcommand{\ol}[1]{\overline{#1}}
\newcommand{\od}{:=}
